\documentclass[10pt]{article}

\usepackage{authblk}
\usepackage[margin=0.9in]{geometry}

\usepackage{graphicx}
\usepackage{multicol,multirow}
\usepackage{amsmath,amssymb,amsfonts}
\usepackage{mathrsfs}
\usepackage[numbers]{natbib}
\usepackage{amsthm}
\newtheorem{theorem}{Theorem}[section]
\newtheorem{lemma}[theorem]{Lemma}

\newtheorem{definition}{Definition}[section]
\newtheorem{proposition}{Proposition}[section]
\newtheorem{remark}[theorem]{Remark}
\newtheorem{example}[theorem]{Example}

\numberwithin{equation}{section}

\title{Weak solvability of nonlinear elliptic equations involving variable exponents}

\author{Ahmed Aberqi$^{1}$, Jaouad Bennouna$^{2}$, Omar Benslimane$^{2}$,\\ Maria Alessandra Ragusa$^{3,4}$\\

$^{1}${\small Laboratory LAMA, Sidi Mohamed Ben Abdellah University, National School of Applied Sciences, Fez, Morocco.}\\

$^{2}${\small Laboratory LAMA, Department of Mathematics, Sidi Mohamed Ben Abdellah University, Faculty of Sciences Dhar El Mahraz, B.P 1796 Atlas Fez, Morocco.}\\

$^{3}${\small Dipartimento di Matematica e Informatica, Università di Catania, Catania, Italy.\protect\\ $^{4}$RUDN University, 6 Miklukho-Maklay St, 117198, Moscow, Russia.}
}

\begin{document}
\maketitle



%
%




\begin{abstract}
We are concerned with the study of the existence and multiplicity of solutions for Dirichlet boundary value problems, involving the $( p( m ), \, q( m ) )-$ equation and the nonlinearity is superlinear but does not fulfil the Ambrosetti-Rabinowitz condition in the framework of Sobolev spaces with variable exponents in a complete manifold. The main results are proved using the mountain pass theorem and Fountain theorem with Cerami sequences. Moreover, an example of a $( p( m ), \, q( m ) )$ equation that highlights the applicability of our theoretical results is also provided.
\end{abstract}

\bigskip
\bigskip
\noindent {\textit{Mathematics Subject Classification(2000):}\thinspace
\thinspace 35J60, 58J05.} \newline


\section[This is an A Head]{Introduction}
Let $ ( \mathrm{M}, g ) $ be a smooth complete compact Riemannian N-manifold. This paper is devoted to the existence and multiplicity of solutions to the following non-linear elliptic problem:
\begin{equation} \label{problem}
\begin{cases}
- \, \Delta_{p( m )} u( m ) - \Delta_{q( m )} u( m ) = \, h( m, u( m ) )  & \text{in \,\,$\mathrm{M}$ }, \\[0.3cm]
\, u \,  = \, 0  & \text{on \,$\partial \mathrm{M}$ },
\end{cases}
\end{equation}
where the variable exponents $p, \, q \in \overline{\mathrm{M}} \rightarrow ( 1, \, + \infty)$ are continuous functions, satisfy the following assumption:
\begin{equation}\label{0}
1 < q^{-} \leq q^{+} < p^{-} \leq p^{+} < \infty,
\end{equation}
with $$ q^{+} = \sup_{m \in \overline{\mathrm{M}}} q( m ), \,\, q^{-} = \inf_{m \in \overline{\mathrm{M}}} q( m ), \,\, p^{+} = \sup_{m \in \overline{\mathrm{M}}} p( m ), \,\, \mbox{and} \,\, p^{-} = \inf_{m \in \overline{\mathrm{M}}} p( m ).$$ The operators $\Delta_{p( m )} u( m ) = - \mbox{div} ( |\, \nabla u( m )\,|^{p( m ) - 2} \nabla u( m ) ) $ and $\Delta_{q( m )} u( m ) = - \mbox{div} ( |\, \nabla u( m )\,|^{q( m ) - 2} \nabla u( m ) ) $ are called the $p( m )$-Laplacian and $q( m )-$Laplacian in $( \mathrm{M}, \, g )$.\\
We set
$$\mathcal{H} = \{ \, G_{\lambda} ( m, s ) = h( m, s ) s - \lambda H( m, s ), \,\, \lambda \in [ 2q^{-}, \, 2p^{+} ] \, \},$$
where $H( m, s ) = \displaystyle\int_{0}^{s} f( m, t )\,\, dt$. \\
We emphasize that throughout this paper, we assume the following hypotheses on the Carathéodory function $ h: \mathrm{M} \times \mathbb{R} \rightarrow \mathbb{R}$
\begin{itemize}
\item[($h_{1})$:] $|\, h( m, s )\,| \leq c_{1} + c_{2} \,|\, s\,|^{r( m ) -1}$ for all $( m, s ) \in \mathrm{M} \times \mathbb{R},$  $r: \overline{\mathrm{M}} \rightarrow (1, \, +\infty )$ is a bounded continuous function with $ r( m ) < q^{*} ( m ) = \frac{N\, q( m )}{N - q( m )}$ for any $ m \in \overline{\mathrm{M}}$.
\item[($h_{2})$:] $\displaystyle\lim_{|\,s\,| \rightarrow \infty} \frac{h( m, s ) \, s}{|\, s\,|^{p^{+}}} = +\infty$ uniformly for a.e $ m \in \mathrm{M}$.
\item[($h_{3})$:] $ h( m, s ) = 0 ( |\, s\,|^{p^{+} - 1} )$ as $ s \rightarrow 0$ uniformly for $ m \in \mathrm{M} $.
\item[($h_{4})$:] There exists a constant $ \theta \geq 1 $ such that for any $ \beta \in [ 0, \, 1], \, s \in \mathbb{R}$ and for all $ G_{\lambda} \in \mathcal{H}, \, \eta \in [ 2q^{-}, \, 2p^{+} ]$ the inequality  $$ \theta \, G_{\lambda} ( m, s ) \geq G_{\eta} ( m, \beta s ) \,\, \mbox{holds for a.e} \,\, m \in \mathrm{M}$$
\item[($h_{5})$:] $ h( m, -s ) = - h( m, s ) $ for all $( m, s ) \in \mathrm{M} \times \mathbb{R}.$
\end{itemize}

In the literature, when $ p = \mbox{cste}$, several papers were done with authors in this case. We start by a pioneer work published by Liu and Li in \cite{liu2003infinitely} for multiplicity results with  superlinear nonlinearities, using the critical point theory with Cerami condition which is weaker than the Palais-Smale condition. For a deeper comprehension, we recommend that readers consult \cite{aberqi2019existence, aberqi2017nonlinear, azroul2021nonlocal, benslimane2020existence1, benslimane2020existence, benslimane2020existence2, benslimane2020existence3} and the references therein.\\
In Sobolev space with variable exponent, some other useful contributions have been devoted to the study but first of all the experts in the field immediately think, among all, to the study made by Colombo and Mingione in \cite{CM} and Baroni, Colombo, and Mingione \cite{BCM}.
The research in the field of differential equations and variational problems has been an attractive topic, which is motivated not only by the study of fluid filtration in porous media, constrained heating, elastoplasticity, optimal control, financial mathematics, and others, see \cite{chen2006variable, gwiazda2008non, ruuvzivcka2004modeling, zhikov2004density} and the references therein. But also by the mathematical importance in the theory of function spaces with variable exponents. For example in \cite{zhang2015existence} Zhang and Zhao proved the existence of strong solutions of the following $p( m )-$Laplacian Dirichlet problem via critical point theory:
$$\begin{cases}
- \, \Delta_{p(m)} v = \, h( m, v( m ) )  & \text{in \,\,$\Omega$ }, \\[0.3cm]
\, v \bigg|_{\partial \Omega}  = \, 0,
\end{cases} $$
and they give a new growth condition, under which, they used a new method to check the Cerami compactness condition. Hence, they proved the existence of strong solutions to the problem as above without the growth condition of the well-known Ambrosetti-Rabinowitz type and also they give some results about the multiplicity of the solutions. By  variational methods, the authors in \cite{hurtado2018existence} established the existence and multiplicity of non-trivial solutions for a general class of quasilinear problems involving $p( \cdot )-$Laplace type operators, without AR- condition,
$$\begin{cases}
- \, \mbox{div} (\mathrm{a} ( | \nabla v |^{p( m )} ) | \nabla v |^{p( m ) - 2} \nabla v = \, \alpha\, h( m, v( m ) )  & \text{in \,\,$\Omega$ }, \\[0.3cm]
\, v \bigg|_{\partial \Omega}  = \, 0,
\end{cases} $$
and using various types of versions of the Mountain Pass Theorem with Cerami sequences, and also, the Fountain and Dual Theorem with Cerami sequences, they obtained some existence of non-trivial solutions for the above problem under some considerations. Furthermore, they have shown that for every parameter $\alpha > 0,$ small enough, the problem treated has at least one non-trivial solution, and also that the solution blows up, in the Sobolev norm, as $ \alpha \rightarrow 0^{+}.$ Finally, utilizing Krasnoselskii's Genus Theory, they obtain the existence of infinitely many weak solutions by presenting appropriate assumptions on the nonlinearity $h( m, . ).$
 We recommend the reader to \cite{ragusa2019regularity, guo2015dirichlet, zang2008p} and the references therein for further information.\\

To our knowledge, the results presented here are new, and they complement and improve the ones obtained in \cite{zang2008p, zhang2015existence, hurtado2018existence} because we are considering the general framework of Sobolev spaces with variable exponents in Complete manifolds and nonlinearities is superlinear but does not fulfil the Ambrosetti-Rabinowitz condition. However, our main difficulty presented by the fact that the $p(m)-$ laplacian and $ q( m )-$ laplacian operators possess more complicated nonlinearities than the  $p$-laplacian and $q-$laplacian operator, on the basis that $\Delta_{p( m )}$ and $ \Delta_{q( m )}$ are not homogeneous. Furthermore, we are unable to apply the Lagrange Multiplier Theorem to a large number of problems using this operator, indicating that our problem is more complicated than the operators of the $p-$Laplace type. \\

This paper is designed as follows. In section \ref{sec2} we will recall the definitions and some properties of Sobolev spaces with variable exponents and Sobolev spaces with variable exponents in a complete manifold. The readers can consult the following papers \cite{aberqi2022fractional, aberqi2021existence, aubin1982nonlinear, benslimane2020existence1, benslimane2020existence, benslimane2020existence2, benslimane2020existence3, hebey2000nonlinear, trudinger1968remarks} for details. In section \ref{sec3},  we establish the existence of non-trivial weak solutions to problem \eqref{problem} applying the mountain pass theorem and Cerami sequences. Furthermore, using the Fountain theorem with Cerami sequences, we demonstrate that the problem \eqref{problem} has infinitely many (pairs) of solutions with unbounded energy.






\section[This is an A Head]{Preliminaries} \label{sec2}
In view of discussing problem \eqref{problem}, we will need some background about spaces $W_{0}^{1, q( m )} ( \Omega ) $ where $\Omega$ is an open subset of $ \mathbb{R}^{N}$ and $ W_{0}^{1, q( m )} ( \mathrm{M} )$ which are well-known as the Sobolev spaces with variable exponents and the Sobolev spaces with variable exponents in a complete manifold. For this reason, we will mention some properties involving the aforementioned spaces, which can be found in \cite{aberqi2021existence, aubin1982nonlinear, benslimane2020existence1, gaczkowski2016sobolev, hebey2000nonlinear, guo2015dirichlet} and the references therein.
\subsection{Sobolev spaces with variable exponents}
Let $ \Omega \subset \mathbb{R}^{N} \, ( N \geq 2 )$ be a bounded open subset. \\
The space $ L^{q(\cdot)} ( \Omega ) = \{ u: \Omega \longmapsto \mathbb{R} \, \, \mbox{measurable}: \, \rho _{q(\cdot)} ( u ) = \displaystyle\int_{\Omega} |\, u( m )\,|^{q( m )} \,\, dm < + \infty\},$ \\
endowed with the norm
$$ ||\, u \,||_{L^{q(\cdot)} ( \Omega )} = \inf \{ \, \alpha > 0: \, \rho_{q(\cdot)} \bigg( \frac{u}{\alpha} \bigg) \leq 1 \, \}, $$
for any $q^{+} = ess \,sup \{ \, q( m ) / m \in \Omega \, \} < + \infty.$ \\
And the Sobolev space
$$ W^{1, q( m )} ( \Omega ) = \{u: \, u \in L^{q( m )} ( \Omega ) \,\, \mbox{and} \,\, |\, \nabla u\,| \in L^{q( m )} ( \Omega ) \,\},$$
endowed  with the norm $$ ||\, u\,||_{W^{1, q( m )} ( \Omega )} = ||\, u\,||_{L^{q( m )} ( \Omega )} + ||\, \nabla u \||_{L^{q( m )} ( \Omega )} \,\, \forall u \in W^{1, q( m )} ( \Omega ).$$
We set $ W^{1, q( m )}_{0} ( \Omega ) = \overline{C_{0}^{\infty} ( \Omega )}^{W^{1, q( m )} ( \Omega )}.$
\subsection{Sobolev spaces with variable exponents in a complete manifold}
\begin{definition}
The Sobolev space $L^{q(\cdot)}_{k} (\mathrm{M})$ is defined as the completion of $C_{k}^{q(\cdot)} (\mathrm{M})$ with respect to the norm $|| u ||_{L^{q(\cdot)}_{k} (\mathrm{M})},$ where $$ C^{q(\cdot)}_{k} ( \mathrm{M} ) = \{ \, u \in C^{\infty} ( \mathrm{M} ) \,\, \mbox{such that } \,\, \forall j \,\, 0 \leq j \leq k \,\, | \, \nabla^{k} u \, | \in L^{q( \cdot ) } ( \mathrm{M} )\, \}, $$
and
$$ || \, u \, ||_{L^{q( \cdot )}_{k}} = \sum_{j = 0}^{k} || \, \nabla^{j} u \, ||_{L^{q( \cdot )}},$$
with $ |\nabla^{k} u| $ being the norm of the $k-$th covariant derivative of $u,$ defined in local coordinates by
$$ | \, \nabla^{k} u \, |^{2} = g^{i_{1} j_{1}} \cdots g^{i_{k} j_{k}} \, ( \nabla^{k} u )_{i_{1} \cdots i_{k}} \, ( \nabla^{k} u )_{j_{1} \cdots j_{k}}. $$
If $\Omega$ is a subset of $\mathrm{M},$ then $L_{k, 0}^{q( \cdot )} ( \Omega )$ is the completion of $C_{k}^{q( \cdot )} ( \mathrm{M} ) \cap C_{0} ( \Omega )$ with respect to $ || \cdot ||_{L^{q( \cdot )}_{k}},$ where $C_{0} ( \Omega )$ denotes the vector space of continuous functions whose support is a compact subset of $\Omega.$
\end{definition}
\begin{definition}
Given $ ( \mathrm{M}, g ) $ a smooth Riemannian manifold, and $ \xi : \, [\, a, \, b \, ] \longrightarrow \mathrm{M} $ a curve of class $ C^{1} $. The length of $ \xi $ is
$$ l( \xi ) = \int_{a}^{b} \bigg( g \, \bigg( \, \frac{d \xi }{d t }, \, \frac{d \xi}{d t}\, \bigg) \bigg)^{\frac{1}{2}} \,\,dt, $$
and for a pair of points $ m, \, \ell \in \mathrm{M}$, we define the distance $ d_{g} ( m, \ell ) $ between $m$ and $\ell$ by
$$ d_{g} ( m, \ell ) = \inf \, \{ \, l( \xi ) : \, \xi: \, [ \, a, \, b \,] \rightarrow M \,\, \mbox{such that} \,\, \xi ( a ) = m \,\, \mbox{and} \,\, \xi ( b ) = \ell \, \}. $$
\end{definition}
\begin{definition}
A function $ \mathcal{S}: \, \mathrm{M} \longrightarrow \mathbb{R} $ is log-Hölder continuous if there exists a constant $c$ such that for every pair of points $ \{ m, \, \ell \} $ in $ \mathrm{M}$ we have
$$ | \, \mathcal{S}( m ) - \mathcal{S}( \ell ) \, | \leq \frac{c}{\log ( e + \frac{1}{d_{g} ( m, \ell )} \, )}. $$
We note with $ \mathcal{P}^{\log} ( \mathrm{M} ) $ the set of log-Hölder continuous variable exponents. The relation between $ \mathcal{P}^{\log} ( \mathrm{M} ) $ and $ \mathcal{P}^{\log} ( \mathbb{R}^{N} ) $ is the following:
\end{definition}
\begin{proposition} (\cite{aubin1982nonlinear, gaczkowski2016sobolev})
Let $ q \in \mathcal{P}^{\log} ( \mathrm{M} ) $, and let $ ( \Omega, \phi ) $ be a chart such that
$$ \frac{1}{2} \delta_{i j } \leq g_{i j} \leq 2 \, \delta_{i j } $$
as bilinear forms, where $ \delta_{i j} $ is the Kronecker delta symbol. Then $ qo\phi^{-1} \in \mathcal{P}^{\log} ( \phi ( \Omega ) ).$
\end{proposition}
\begin{definition}
We say that the N-manifold $ ( \mathrm{M}, g ) $ has property $ B_{vol} ( \lambda, v ) $ if its geometry is bounded in the following sense:\\
$ \hspace*{1cm} \bullet \,\, \mbox{The Ricci tensor of g noted by Rc ( g ) verify,} \,\,Rc ( g ) \geq \lambda ( N - 1 ) \, g $ for some $ \lambda. $\\
$ \hspace*{1cm} \bullet  $ There exists some $ v > 0 $ such that $ | \, B_{1} ( m ) \, |_{g} \geq v \,\, \forall m \in \mathrm{M},$ where $B_{1} ( m ) $ are the balls of radius 1 centred at some point $m$ in terms of the volume of smaller concentric balls.
\end{definition}
\label{prop5}
\begin{proposition} (\cite{aubin1982nonlinear,hebey2000nonlinear})\label{prop3}
Assume that the complete n-manifold $ ( \mathrm{M}, g ) $ has property $ B_{vol} ( \lambda, v ) $ for some $ ( \lambda, v ).$ Then there exist positive constants $ \delta_{0} = \delta_{0} ( N, \, \lambda, \, v ) $ and $ A = A ( N, \, \lambda, \, v ) $, we have, if $ R \leq \delta_{0} $, if $ m \in \mathrm{M} $ if $ 1 \leq q \leq N $, and if $ u \in L^{q}_{1,0} ( \, B_{R} ( m ) \, ) $ the estimate
$$ || \, u \, ||_{L^{p}} \leq A \,p \, || \, \nabla u \, ||_{L^{q}},$$ where $ \frac{1}{p} = \frac{1}{q} - \frac{1}{N}.$
\end{proposition}
\begin{proposition} (\cite{aubin1982nonlinear,hebey2000nonlinear, gaczkowski2016sobolev}) \label{prop6}
Suppose that for some $ ( \lambda, v ) $ the complete n-manifold $( \mathrm{M}, g ) $ has property $ B_{vol} ( \lambda, v ) $. Let $ p \in \mathcal{P} ( \mathrm{M} ) $ be uniformly continuous with $ q^{+} < N.$ Then $ L^{q( \cdot )}_{1} ( \mathrm{M} ) \hookrightarrow L^{p( \cdot )} ( \mathrm{M} ) \,\, \forall q \in \mathcal{P} ( \mathrm{M} ) $ such that $ q \ll p \ll q* = \frac{N q}{N - q}. $ Indeed, for $ || \, u \, ||_{L^{q( \cdot )}_{1}} $ sufficiently small, we have $$ \rho_{p( \cdot )} ( u ) \leq G \, ( \, \rho_{q( \cdot )} ( u ) + \rho_{q( \cdot )} ( | \, \nabla u \, | ) \, ), $$ with $G$ is a positive constant depend on $ n, \, \lambda, \, v, \, q $ and $ p $.
\end{proposition}
\begin{definition} (\cite{guo2015dirichlet})
The Sobolev space $ W^{1, q( m )} ( \mathrm{M} )$ consists of such functions $ u \in L^{q( m )} ( \mathrm{M} )$ for which $ \nabla^{k} u \in L^{q( m )} ( \mathrm{M} )$ $k = 1, 2, \cdots, n.$ The norm is defined by
$$ ||\, u\,||_{W^{1, q( m )} ( \mathrm{M} )} = ||\, u\,||_{L^{q( m )} ( \mathrm{M} )} + \sum_{k = 1}^{n} ||\, \nabla^{k} u \,||_{L^{q( m )} ( \mathrm{M} )}.$$
The space $ W_{0}^{1, q( m )} ( \mathrm{M} )$ is defined as the closure of $ C^{\infty} ( \mathrm{M} ) $ in $ W^{1, q( m )} ( \mathrm{M} ).$
\end{definition}
\begin{theorem}\label{theo1} (\cite{guo2015dirichlet})
Let $\mathrm{M}$ be a compact Riemannian manifold with a smooth boundary or without boundary and $ q( m ), \, p( m ) \in C( \overline{\mathrm{M}} ) \cap L^{\infty} ( \mathrm{M} ).$ Assume that $$ q( m ) < N , \hspace*{0.5cm} p( m ) < \frac{N\, q( m )}{N - q( m )} \,\, \mbox{for} \,\, m \in \overline{\mathrm{M}}.$$
Then, $$ W^{1, q( m )} ( \mathrm{M} ) \hookrightarrow L^{p( m )} ( \mathrm{M} ),$$
is a continuous and compact embedding.
\end{theorem}
\begin{theorem}\label{theo2}
Let $\mathrm{M}$ be a compact Riemannian manifold with a smooth boundary or without boundary and $ q( m ), \, r( m ) \in C( \overline{\mathrm{M}} ) \cap L^{\infty} ( \mathrm{M} ).$ Assume that $$ q( m ) < N , \hspace*{0.5cm} r( m ) < \frac{N\, q( m )}{N - q( m )} \,\, \mbox{for} \,\, m \in \overline{\mathrm{M}}.$$
Then, $$ W^{1, q( m )} ( \mathrm{M} ) \hookrightarrow L^{r( m )} ( \mathrm{M} ),$$
is a continuous and compact embedding.
\end{theorem}
\begin{proof}
The demonstration of this theorem is the same as the previous one.
\end{proof}
\begin{proposition} (\cite{aubin1982nonlinear})
If $( \mathrm{M}, g )$ is complete, then  $W^{1, q( m )} ( \mathrm{M} ) = W^{1, q( m )}_{0} ( \mathrm{M} ).$
\end{proposition}

Let $ D \Psi = D \mathcal{J} - DI: \, L^{q( m )}_{1} ( \mathrm{M} ) \cap L^{p( m )}_{1} ( \mathrm{M} ) \rightarrow \mbox{Hom} ( L^{q( m )}_{1} ( \mathrm{M} ) \cap L^{p( m )}_{1} ( \mathrm{M} ), \, \mathbb{R} )$ the differential of $ \Psi = \mathcal{J} - I $ with the functional $\mathcal{J}: \mathcal{X} \rightarrow \mathbb{R} $ defined by:
$$\mathcal{J} ( u ) = \int_{\mathrm{M}} \frac{1}{p( m )} |\, \nabla u( m )\,|^{p( m )} \,\, dv_{g} ( m ) + \int_{\mathrm{M}} \frac{1}{q( m )} |\, \nabla u( m )\,|^{q( m )} \,\, dv_{g} ( m )$$
and $$ I ( u ) = \int_{\mathrm{M}} H( m, u( m ) ) \,\, dv_{g} ( m ).$$
Where $\mathcal{X} = W_{0}^{1, q( m )} ( \mathrm{M} )\, \cap \, W_{0}^{1, p( m )} ( \mathrm{M} ),$ with the norm $ || u ||_{\mathcal{X}} = || u ||_{q( m )} + || u ||_{p( m )} \,\, \mbox{for all} \,\, m \in \mathrm{M}.$
Then, for all $\varphi \in D( \mathrm{M} )$ we have
$$\langle \, \mathcal{J} ( u ), \, \varphi \,\rangle = \int_{\mathrm{M}} |\, \nabla u( m )\,|^{p( m ) - 2} u( m ) \, \varphi ( m ) \,\, dv_{g} ( m ) + \int_{\mathrm{M}} |\, \nabla u( m ) \,|^{q( m ) - 2} u( m )\, \varphi ( m )\,\, dv_{g} ( m ).$$
\begin{lemma}
The following assumptions hold:
\begin{itemize}
\item[i/] $\mathcal{J}$ is a continuous, bounded homeomorphism and strictly monotone operator.
\item[ii/] $\mathcal{J}$ is of type $( \mathcal{S}_{+} )$, i.e. if $ u_{n} \rightharpoonup u$ and $$\lim_{n \rightarrow + \infty} \, \sup \langle \, \mathcal{J} ( u_{n} ) - \mathcal{J} ( u ), \, u_{n} - u \, \rangle \, \leq 0,$$ then, $u_{n} \longrightarrow u. $
\end{itemize}
\end{lemma}
We recall G. Cerami's definition of Cerami sequences $( Cer )$ in \cite{cerami1978existence}.
\begin{definition} (\cite{cerami1978existence})\label{cer}
Given $( \mathcal{E}, \, ||\,\cdot\,|| )$ a Banach space  and $\Phi \in C^{1} ( \mathcal{E}, \, \mathbb{R} ).$ Let $ c \in \mathbb{R},$  $\Phi $ is said to satisfy the Cerami $(Cer)$ condition (we denote condition $( Cer_{c} )$), if:
\begin{itemize}
\item[$( Cer_{1} ) $:] Every bounded sequence ${\, u_{n} \,} \subset \mathcal{E}$ such as $ \Phi ( u_{n} ) \rightarrow c $ and $\Phi'( u_{n} ) \rightarrow 0$ has a convergent subsequence.
\item[$( Cer_{2} )$ :] There exist constant $ \eta, \, \alpha,\, \beta > 0$ such as $$ ||\, \Psi'( u ) \,||_{\mathcal{E}^{*}} \, ||\, u\,|| \geq \beta \,\, \mbox{for all} \,\, u \in \Psi^{-1} ( [\, c - \eta, \, c + \eta \,] ) \,\, \mbox{with} \,\, ||\, u\,|| \geq \alpha.$$
\end{itemize}
\end{definition}
We say that $ \Psi $ satisfies condition $( Cer )$, if $ \Psi \in C^{1} ( \mathcal{X}, \, \mathbb{R} )$ satisfies condition $( Cer_{c} )$ for every $ c \in \mathbb{R}.$ \\
Let's review the following version of the mountain pass Lemma with Cerami sequences, which will be applied in the next section.
\begin{proposition} (\cite{cerami1978existence})\label{prop1}
Given $ ( \mathcal{E}, \, ||\cdot|| )$ a Banach space, $\Psi \in C^{1} ( \mathcal{X}, \, \mathbb{R} ), \,\, u_{0} \in \mathcal{E}$ and $ \upsilon > 0,$ such that $ || u_{0} || > \upsilon$ and $$ b = \inf_{|| u_{0} || = \upsilon} \Psi ( u ) > \Psi ( 0 ) \geq \Psi ( u_{0} ).$$
If $\Psi $ satisfies the condition $( Cer_{c} )$ with $$ c = \inf_{\nu \in \Gamma} \max_{t \in [ 0, \, 1 ]} \Psi ( \nu ( t ) ), \,\,\, \Gamma = \{ \nu \in C( [ 0, \, 1 ],\, \mathcal{X} ) \,\mbox{such as} \,\, \nu ( 0 ) = 0, \, \nu ( 1 ) = u_{0} \,\}.$$
Then $c$ is a critical value of $\Psi.$
\end{proposition}
\begin{remark}\label{remark1}
As $\mathcal{X}$ is a reflexive and separable Banach space. Then, there exist $\{ \mathrm{e}_{j} \}_{i = 1}^{\infty} \subset \mathcal{X}$ and $\{ \mathrm{e}^{*}_{i} \}_{i = 1}^{\infty} \subset \mathcal{X}^{*}$ such that
$$ \langle \mathrm{e}^{*}_{i}, \, \mathrm{e}_{j} \rangle \,= \delta_{i,\,j},$$ with $ \delta_{i,\,j}$ is the Kronecker delta symbol.
Hence, $$\mathcal{X} = \overline{\mbox{span}} \, \{ \mathrm{e}_{i}, \, i \geq 1 \,\} \,\, \mbox{and} \,\,\mathcal{X}^{*} = \overline{\mbox{span}} \, \{ \mathrm{e}^{*}_{i}, \, i \geq 1 \,\}.  $$
Let $\mathcal{X}^{k} = \mbox{span} \,\{\, \mathrm{e}_{k} \,\}, \,\, Y_{k} = \displaystyle \bigoplus_{i = 0}^{k} \mathcal{X}^{i}, \,\, \mbox{and} \,\, Z_{k} = \displaystyle \overline{\bigoplus_{i = k}^{\infty} \mathcal{X}^{i}},$ for any $k \geq 1.$
\end{remark}
\begin{theorem} (Fountain Theorem, \cite{zou2001variant}) \label{prop2}
Suppose that $\mathcal{X}$ is a separable Banach space, $\Psi \in C^{1} ( \mathcal{X}, \, \mathbb{R} )$ is an even functional satisfying the Cerami sequences $( Cer ).$ Moreover, for all $ k \geq 1$ there exist $ D_{k} > d_{k} > 0$ such as
\begin{itemize}
\item[$( A_{1} )$:] $\displaystyle \lim_{k \rightarrow + \infty} \inf_{\{ u \in Z_{k} : || u || = d_{k} \}} \Psi ( u ) = + \infty.$
\item[$( A_{2} )$:] $ \displaystyle \max_{\{ u \in Y_{k} : || u || = D_{k} \}} \Psi ( u ) \leq 0.$
\end{itemize}
Then, $\Psi$ has a sequence of critical values which tends to $+ \infty.$
\end{theorem}

\section{Existence and multiplicity of non-trivial solutions}\label{sec3}

In this part, we give our main results and we note by $D( \mathrm{M} )$ the space of $C^{\infty}$ functions with compact support in $\mathrm{M}$.
\begin{definition}
We said that $u \in \mathcal{X}$ be a non-trivial solution of the problem \eqref{problem} if
\begin{align*}
&\int_{\mathrm{M}} | \nabla u( m )|^{p( m ) - 2} g( \nabla u( m ), \, \nabla \phi ( m ) ) \,\, dv_{g} ( m ) + \int_{\mathrm{M}} | \nabla u( m ) |^{q( m ) - 2} g( \nabla u( m ), \nabla \phi ( m ) ) \,\, dv_{g} ( m )\\& - \int_{\mathrm{M}} h( m, u( m ) ) \,.\, \phi ( m ) \,\, dv_{g} ( m ) = 0,  \,\,\, \forall \, \phi \in D( \mathrm{M} ).
\end{align*}

Defining the energy function $ \Psi : \mathcal{X} \rightarrow \mathbb{R} $ by
\begin{align*}
\Psi ( u ) =& \int_{\mathrm{M}} \frac{1}{p( m )} | \nabla u( m )|^{p( m )} \,\, dv_{g} ( m ) + \int_{\mathrm{M}} \frac{1}{q( m )} | \nabla u( m ) |^{q( m )} \,\, dv_{g} ( m )\\& - \int_{\mathrm{M}} H( m, u( m ) ) \,\, dv_{g} ( m ).
\end{align*}
By condition $( h_{1} )$ and Theorems \ref{theo1} and \ref{theo2}, $ \Psi \in C^{1} ( \mathcal{X}, \, \mathbb{R} )$ is well-defined.\\
Moreover, for all $ \phi \in D( \mathrm{M} )$ we have
\begin{align*}
\langle \Psi'( u ), \phi \rangle \, = &\int_{\mathrm{M}} | \nabla u( m )|^{p( m ) - 2} g( \nabla u( m ), \, \nabla \phi ( m ) )\,\, dv_{g} ( m ) \\&+ \int_{\mathrm{M}} | \nabla u( m )|^{q( m ) - 2} g( \nabla u( m ), \, \nabla \phi ( m ) )\,\, dv_{g} ( m ) \\&- \int_{\mathrm{M}} h( m, u( m ) ) \,.\, \phi ( m ) \,\, dv_{g} ( m ) \hspace*{1cm} \forall u \in \mathcal{X}
\end{align*}
\end{definition}

\begin{lemma}\label{lemma2}
Assume that the assumptions $(h_{1} ), \, ( h_{2} ),$ and $( h_{4} )$ are satisfied. Then the functional $\Psi$ fulfils $(Cer_{1} ) - (Cer_{2} )$ of Definition \ref{cer}.
\end{lemma}
\begin{proof}
Firstly, we prove that $ \Psi$ satisfies the assertion $( Cer_{1} )$.
Letting $c \in \mathbb{R},$and $\{ u_{m} \} \subset \mathcal{X}$ be a bounded sequence such as
\begin{equation}\label{1}
\begin{cases}
\Psi ( u_{j} ) \xrightarrow{j \rightarrow + \infty} c, \\[0.3cm]
\Psi' ( u_{j} ) \xrightarrow{j \rightarrow + \infty} 0,
\end{cases}
\end{equation}
Then, we can extract a subsequence $\{u_{j}\},$ with $ u_{j} \rightharpoonup u$ as $ j \rightarrow + \infty.$ Using \eqref{1} we get
\begin{align}\label{2}
\big \langle \Psi' ( u_{j} ), \,u_{j} - u \big \rangle \, =& \int_{\mathrm{M}} | \nabla u_{j} |^{p( m ) -2} \nabla u_{j} \, ( \nabla u_{j} - \nabla u ) \,\, dv_{g} ( m ) \nonumber \\&+ \int_{\mathrm{M}} | \nabla u_{j} |^{q( m ) - 2} \nabla u_{j} \, ( \nabla u_{j} - \nabla u ) \,\, dv_{g} ( m ) \nonumber \\&- \int_{\mathrm{M}} h( m, u_{j} ( m ) )\, ( u_{j} - u )\,\, dv_{g} ( m ) \xrightarrow{j \rightarrow + \infty} 0,
\end{align}
Thanks to $( h_{1} )$ and Hölder inequality, we get
\begin{equation}\label{3}
\int_{\mathrm{M}} h( m, u_{j} ( m ) ) \, ( u_{j} - u ) \,\, dv_{g} ( m ) \xrightarrow{j \rightarrow + \infty} 0.
\end{equation}
Combining \eqref{2} and \eqref{3}, we get
\begin{align*}
&\int_{\mathrm{M}} | \nabla u_{j} ( m )|^{p( m ) - 2} \, \nabla u_{j} \, ( \nabla u_{j} - \nabla u ) \,\, dv_{g} ( m ) \\& + \int_{\mathrm{M}} | \nabla u_{j} |^{q( m ) - 2} \, \nabla u_{j} \, ( \nabla u_{j} - \nabla u ) \,\, dv_{g} ( m ) \rightarrow 0 \,\, \mbox{as} \,\, j \rightarrow + \infty.
\end{align*}
That is
\begin{equation}\label{4}
\big \langle \mathcal{J} ( u_{j} ), \, u_{j} - u \, \big \rangle \rightarrow 0 \,\, \mbox{as} \, \, j \rightarrow + \infty.
\end{equation}
Furthermore, since $ u_{j} \rightharpoonup u$ as $ j \rightarrow + \infty,$ from \eqref{1} we have $$ \big \langle \Psi'( u_{j} ), \, u_{j} - u \big \rangle \rightarrow 0 \,\, \mbox{as} \,\, j \rightarrow + \infty.$$
Using the same technique as before, we deduce that
\begin{equation}\label{5}
\big \langle \mathcal{J}( u ), \, u_{j} - u \big \rangle \rightarrow 0 \,\, \mbox{as} \,\, j \rightarrow + \infty.
\end{equation}
Hence, according to \eqref{4} and \eqref{5} we deduce that
$$ \lim_{j \rightarrow + \infty} \sup \, \big \langle \mathcal{J} ( u_{j} ) - \mathcal{J} ( u ), \, u_{j} - u \, \big \rangle \, \leq 0. $$
Thus, since $\mathcal{J}$ is of type $ ( \mathcal{S}_{+} )$ and $ u_{j} \rightharpoonup u$ in $\mathcal{X},$ we conclude that $u_{j} \xrightarrow{j \rightarrow + \infty} u$ in $\mathcal{X}$.\\
Now, we prove that $\psi$ satisfies $( Cer_{2} )$. Arguing by contradiction, there exist $ c \in \mathbb{R}$ and $\{ u_{j} \} \subset \mathcal{X}$ satisfying:
\begin{equation} \label{6}
\Psi ( u_{j} ) \xrightarrow{j \rightarrow + \infty} c, \hspace*{0.2cm} || u_{j} ||_{\mathcal{X}} \xrightarrow{j \rightarrow + \infty} + \infty, \hspace*{0.2cm} || \Psi'( u_{j} )||_{\mathcal{X}^{*}} \, || u_{j} ||_{\mathcal{X}} \xrightarrow{j \rightarrow + \infty} 0.
\end{equation}\label{7}
Let $$ \alpha_{j} = \frac{\displaystyle \int_{\mathrm{M}} \big( | \nabla u_{j} |^{p( m )} + | \nabla u_{j} |^{q( m )} \, \big) \,\, dv_{g} ( m )}{J' ( u_{j} )},$$
by choosing $ || u_{j} ||_{\mathcal{X}} > 1,$ for $ j \in \mathbb{N},$ we get
\begin{align}\label{8}
c &= \lim_{j \rightarrow + \infty} \{ \Psi ( u_{j} ) - \frac{1}{\alpha_{j}} \, \langle \Psi'( u_{j} ), \, u_{j} \, \rangle \, \} \nonumber \\&= \lim_{j \rightarrow + \infty} \{ \, \frac{1}{\alpha_{j}} \, \int_{\mathrm{M}} h( m, u_{j} ( m ) ) \,.\, u_{j} \,\, dv_{g} ( m ) - \int_{\mathrm{M}} F( m, u_{j} ( m ) ) \,\, dv_{g} ( m ) \, \}.
\end{align}
Denote $ w_{j} = \frac{u_{j}}{|| u_{j} ||}, $ so $|| w_{j} ||_{\mathcal{X}} = 1,$ which implies that $\{ u_{j} \}$ is bounded in $\mathcal{X}$. \\
Thus, for a subsequence of $\{ u_{j} \}$ still denoted by $\{ w_{j} \},$ and $ w \in \mathcal{X},$ we get
\begin{equation}\label{9}
w_{j} \rightharpoonup w \hspace*{0.3cm} \mbox{in} \,\,\, \mathcal{X},
\end{equation}
\begin{equation}\label{10}
w_{j} \rightarrow w \hspace*{0.3cm} \mbox{in} \,\,\, L^{r( m )} ( \mathrm{M} )
\end{equation}
\begin{equation}\label{11}
w_{j} ( m ) \rightarrow  w( m ) \hspace*{0.3cm} \mbox{a.e in} \,\,\, \mathrm{M}
\end{equation}
\textbf{Step 1:} \underline{If $ w = 0$:} We proceed as in \cite{jeanjean1999existence}, let $\{ t_{j} \} \subset [ 0, \, 1 ]$ such as
\begin{equation}\label{12}
\psi ( t_{j} u_{j} ) = \max_{t \in [ 0, \, 1 ]} \psi ( t u_{j} ).
\end{equation}
If for $ j \in \mathbb{N}, \, t_{j}$ satisfying \eqref{12} is not unique, then we choose the smaller positive value. For that, we fix $ A > \frac{1}{2 p^{+}},$ let $ \tilde{w}_{j} = ( 2 p^{+} A )^{\frac{1}{p^{-}}},$ and according to \eqref{10} we have that
$$ \tilde{w}_{j} \rightarrow 0 \,\, \mbox{in} \,\, L^{r( m )} ( \mathrm{M} ),$$
and by $( h_{1} ),$ we have $$|\, H( m, t ) \,| \leq c \, ( 1 + |\, t\,|^{r( m )} ).$$
Since the Nemitskii operator is continuous, we have $$ H( ., \tilde{w}_{j} ) \rightarrow 0 \,\, \mbox{in} \,\, L^{1} ( \mathrm{M} ) \,\, \mbox{as} \,\, j \rightarrow + \infty.$$
Therefore,
\begin{equation}\label{13}
\lim_{j \rightarrow + \infty} \int_{\mathrm{M}} H( m, \tilde{w}_{j} ) \,\, dv_{g} ( m ) = 0.
\end{equation}
Then, for $j$ large enough, $$ \frac{( 2 p^{+} A )^{\frac{1}{p^{-}}}}{|| u_{j} ||_{X}} \in ( 0, \, 1 ), $$ and
\begin{align*}
\Psi ( t_{j} u_{j} ) &\geq \, \Psi ( \tilde{w}_{j} ) \\& \geq \int_{\mathrm{M}} \frac{1}{p( m )} | \nabla \tilde{w}_{j} |^{p( m )} \,\, dv_{g} ( m ) + \int_{\mathrm{M}} \frac{1}{q( m )} | \nabla \tilde{w}_{j} |^{q( m )} \,\, dv_{g} ( m ) \\& \hspace*{0.5cm} - \int_{\mathrm{M}} H( m, \tilde{w}_{j} ) \,\, dv_{g} ( m ) \\& \geq \frac{1}{p^{+}} \, \int_{\mathrm{M}} ( 2 p^{+} A ) \, | \nabla w_{j} |^{p( m )} \,\, dv_{g} ( m ) + \frac{1}{q^{+}} \int_{\mathrm{M}} ( 2 p^{+} A ) \, | \nabla w_{j} |^{q( m )} \,\, dv_{g} ( m )\\& \hspace*{0.5cm}   - \int_{\mathrm{M}} H( m, \tilde{w}_{j} ) \,\, dv_{g} ( m ) \\& \geq 2A \int_{\mathrm{M}} | \nabla w_{j} |^{p( m )} \,\, dv_{g} ( m ) + \frac{2 A p^{+}}{q^{+}} \int_{\mathrm{M}} | \nabla w_{j} |^{q( m )} \,\, dv_{g} ( m )\\& \hspace*{0.5cm}  - \int_{\mathrm{M}} H( m, \tilde{w}_{j} ) \,\, dv_{g} ( m ) \\& \geq 2 A c || w_{j} ||^{p^{+}} + \frac{2 A}{\eta \, q^{+}} || w_{j} ||^{p^{+}} - \int_{\mathrm{M}} H( m, \tilde{w}_{j} ) \,\, dv_{g} ( m ) \\& \geq A.
\end{align*}
That is  \begin{equation}\label{14}
\Psi ( t_{j} u_{j} ) \rightarrow + \infty.
\end{equation}
As $ \Psi ( u_{j} ) \xrightarrow{j \rightarrow + \infty} c$ and $ \Psi ( 0 ) = 0 $, we have $ t_{j} \in ( 0, \, 1 )$ for $j$ large enough, and
\begin{align}\label{15}
&\int_{\mathrm{M}} | \nabla ( t_{j} u_{j} ) |^{p( m )} \,\, dv_{g} ( m ) + \int_{\mathrm{M}} | \nabla t_{j} u_{j} ) |^{q( m )} \,\, dv_{g} ( m ) - \int_{\mathrm{M}} h( m, t_{j} u_{j} ) \,\, dv_{g} ( m ) \nonumber \\&= \langle \Psi'( t_{j} u_{j} ), \, t_{j} u_{j} \, \rangle = t_{j} \,\frac{\mbox{d}}{\mbox{dt}} \bigg \vert _{t = t_{j}} \Psi ( t u_{j} ) = 0.
\end{align}
Thus, from \eqref{14} and \eqref{15}, we get
\begin{align*}
&\int_{\mathrm{M}} \bigg( \, \frac{1}{\alpha_{t_{j}}} h( m, \, t_{j} u_{j} ) \, t_{j} u_{j} - H( m, \, t_{j} u_{j} ) \, \bigg) \,\, dv_{g} ( m )\\& = \frac{1}{\alpha_{t_{j}}} \int_{\mathrm{M}} | \nabla t_{j} u_{j} |^{p( m )} \,\, dv_{g} ( m ) + \frac{1}{\alpha_{t_{j}}} \int_{\mathrm{M}} | \nabla ( t_{j} u_{j} ) |^{p( m )} \,\, dv_{g} ( m ) \\& \hspace*{0.3cm}+ \frac{1}{\alpha_{t_{j}}} \int_{\mathrm{M}} | \nabla ( t_{j} u_{j} ) |^{q( m )} \,\, dv_{g} ( m ) - \int_{\mathrm{M}} H( m, t_{j} u_{j} ) \,\, dv_{g} ( m ) \\& = \Psi ( t_{j} u_{j} ) \xrightarrow{j \rightarrow + \infty} + \infty,
\end{align*}
where, $$ \alpha_{t_{j}} = \frac{\displaystyle\int_{\mathrm{M}} \bigg( \, | \nabla ( t_{j} u_{j} ) |^{p( m )} + | \nabla ( t_{j} u_{j} ) |^{q( m )} \, \bigg) \,\, dv_{g} ( m )}{J' ( t_{j} u_{j} )}.$$
From the definition of $\alpha_{j}$ and $\alpha_{t_{j}}$, we have $ \alpha_{j}, \, \alpha_{t_{j}} \in [ 2 q^{-}, \, 2 p^{+} ].$ Hence, $ G_{\alpha_{j}}, \, G_{\alpha_{t_{j}}} \in \mathcal{H}.$ Then, according to $( h_{4} )$ and the fact that $$ \frac{\alpha_{t_{j}}}{\theta \, \alpha_{j}} > 0,$$ we deduce that
\begin{align*}
&\int_{\mathrm{M}} \bigg( \, \frac{1}{\alpha_{j}} h( m, u_{j} ) \, u_{j} - H( m, u_{j} ) \, \bigg) \,\, dv_{g} ( m ) \\&= \frac{1}{\alpha_{j}} \, \int_{\mathrm{M}} G_{\alpha_{j}} ( m, u_{j} ) \,\, dv_{g} ( m )\\& \geq \frac{1}{\theta \, \alpha_{j}} \int_{\mathrm{M}} G_{\alpha_{t_{j}}} ( m, t_{j} u_{j} ) \,\, dv_{g} ( m ) \\& = \frac{\alpha_{t_{j}}}{\theta \alpha_{j}} \int_{\mathrm{M}} \bigg( \frac{1}{\alpha_{t_{j}}} h( m, t_{j} u_{j} ) \, t_{j} u_{j} - H( m, t_{j} u_{j} ) \, \bigg) \,\, dv_{g} ( m ) \longrightarrow + \infty,
\end{align*}
which contradicts \eqref{8}.\\\\
\textbf{Step 2:} \underline{If $ w \neq 0$:} From \eqref{7} we write
\begin{align}\label{16}
&\int_{\mathrm{M}} | \nabla u_{j} |^{p( m )} \,\, dv_{g} ( m ) + \int_{\mathrm{M}} | \nabla u_{j} |^{q( m )} \,\, dv_{g} ( m ) - \int_{\mathrm{M}} h( m, u_{j} ) \, u_{j} \,\, dv_{g} ( m ) \nonumber \\& = \langle \Psi'( u_{j} ), \, u_{j} \rangle = o( 1 ) \, || u_{j} ||_{\mathcal{X}},
\end{align}
then,
\begin{align}
1 - o( 1 ) &= \int_{\mathrm{M}} \frac{h( m, u_{j} )\,.\, u_{j}}{\displaystyle \int_{\mathrm{M}} | \nabla u_{j} |^{p( m )} \,\, dv_{g} ( m ) + \displaystyle \int_{\mathrm{M}} | \nabla u_{j} |^{q( m )} \,\, dv_{g} ( m )} \,\, dv_{g} ( m ) \nonumber \\& \geq \int_{\mathrm{M}} \frac{h( m, u_{j} ) \, u_{j}}{|| u_{j} ||^{p^{+}}} \,\, dv_{g} ( m )\nonumber \\& = \int_{\mathrm{M}} \frac{h( m, u_{j} ) \,.\, u_{j}}{| u_{j} |^{p^{+}}} \,.\, | w_{j} |^{p^{+}} \,\, dv_{g} ( m ).
\end{align}
Considering the set $ \mathcal{B} = \{ \, m \in \mathrm{M}; \, w( m ) = 0 \,\}.$ For any $ m \in \mathcal{B} \backslash \mathcal{B}_{0} = \{ \, m \in \mathrm{M} ; \, w( m ) \neq 0 \, \},$ we get
$$ | u_{j} ( m ) | \longrightarrow + \infty \,\, \mbox{as} \,\, j \rightarrow + \infty.$$
Then, by $( h_{4} )$ we have
\begin{equation}\label{17}
\frac{h( m, u_{j} ( m ) ) \,.\, u_{j} ( m )}{| u_{j} ( m ) |^{p^{+}}} \,.\, | w_{j} ( m ) |^{p^{+}} \longrightarrow + \infty \,\, \mbox{as} \,\, m \rightarrow + \infty.
\end{equation}
Since, $| \mathcal{B} \backslash \mathcal{B}_{0} | > 0,$ Fatou's Lemma allows us
\begin{equation}\label{18}
\int_{\mathcal{B} \backslash \mathcal{B}_{0}} \frac{h( m, \, u_{j} ) \,.\, u_{j}}{| u_{j} |^{p^{+}}} \,.\, | w_{j} |^{p^{+}} \,\, dv_{g} ( m ) \longrightarrow + \infty \,\, \mbox{as} \,\, j \rightarrow + \infty.
\end{equation}
And, from $( h_{1} )- ( h_{4} ),$ there exists $ l > - \infty$ such as $$ \frac{h( m, t )\,.\, t}{| t |^{p^{+}}} \geq l \,\, \mbox{for} \,\, t \in \mathbb{R} \,\, \mbox{and  a.e} \,\, m \in \mathrm{M}.$$ Furthermore, we have  $$ \int_{\mathcal{B}_{0}} | w_{j} ( m ) |^{p^{+}} \,\, dv_{g} ( m ) \longrightarrow 0.$$
Thus, there exists $ \mathcal{K} > - \infty$ such as
\begin{equation}\label{19}
\int_{\mathcal{B}_{0}} \frac{h( m, u_{j} ) \,.\, u_{j}}{| u_{j} |^{p^{+}}} \,.\, | w_{j} |^{p^{+}} \,\, dv_{g} ( m ) \geq l \, \int_{\mathcal{B}_{0}} | w_{j} |^{p^{+}} \,\, dv_{g} ( m ) \geq  \mathcal{K} > - \infty.
\end{equation}
A contradiction is obtained by combining \eqref{17} - \eqref{19}. Therefore, $( Cer_{2} )$ is fulfilled, which completes the proof.
\end{proof}
Now, we will demonstrate our first existence theorem.
\begin{theorem}\label{theo3}
Assume that $( h_{1} ) - ( h_{4} )$ are satisfied, and we suppose that the smooth complete compact Riemannian N-manifold $( \mathrm{M}, \, g )$ has property $B_{vol} ( \lambda, \, v ).$ If $ q^{+} < p^{-},$ then the problem \eqref{problem} has at least one non-trivial solution.
\end{theorem}
\begin{proof}
By Lemma \ref{lemma2}, $\Psi$ satisfies $( Cer )$ on $\mathcal{X}$. First, we prove that the functional $\Psi$ has a geometrical structure, in order to apply Proposition \ref{prop1}. For that, we claim that there exists $\mu, \, \nu > 0$ such as $$ \Psi ( u ) \geq \mu > 0 \,\, \mbox{for any} \,\, u \in \mathcal{X} \,\, \mbox{with} \,\, || u ||_{\mathcal{X}} = \nu.$$
Let $ || u ||_{\mathcal{X}} < 1.$ Then by Proposition \ref{prop5}, and the fact that $ q^{+} < p^{+} $ we get
\begin{align}\label{20}
\Psi ( u ) &\geq \frac{c}{p^{+}} || u ||^{p^{+}}  + \frac{1}{A p^{-} q^{+}} || u ||^{q^{+}} - \int_{\mathrm{M}} H( m, u ) \,\, dv_{g} ( m ) \nonumber \\& \geq c^{*} || u ||^{q^{+}}_{\mathcal{X}} - \int_{\mathrm{M}} H( m, u( m ) ) \,\, dv_{g} (m),
\end{align}
 with $ c^{*} = \min \, \{ \frac{c}{p^{+}}, \, \frac{1}{Ap^{-} q^{+}} \, \}.$
 According to theorems \ref{theo1} and \ref{theo2}, there exist two positive constants $c_{1}, \, c_{2} > 0$ such as $$ | u |_{p^{+}} \leq c_{1} || u ||_{\mathcal{X}} \,\, \mbox{and} \,\, | u |_{r( m )} \leq || u ||_{\mathcal{X}} \,\, \mbox{for any} \,\, u \in \mathcal{X}.$$
 Let $ \epsilon > 0$ be small enough, such as $$ \epsilon \, c_{1}^{p^{+}} < \frac{c^{*}}{2}.$$
 According to $( h_{1} )$ and $( h_{2} )$, we have $$ H( m, t ) \leq \epsilon \, | t |^{p^{+}} + c_{\epsilon} \, | t |^{r( m )} \,\, \mbox{for all} \,\, ( m, t ) \in \mathrm{M} \times \mathbb{R},$$
 for $|| u || \leq 1,$ we get
 \begin{align*}
 \Psi ( u ) & \geq c^{*} \, || u ||^{q^{+}}_{\mathcal{X}} - \epsilon \, \int_{\mathrm{M}} | u |^{p^{+}} \,\, dv_{g} ( m ) - c_{\epsilon} \, \int_{\mathrm{M}} | u( m ) |^{r( m )} \,\, dv_{g} ( m ) \\& \geq c^{*} \, || u ||_{\mathcal{X}}^{q^{+}} - \epsilon \, || u ||_{p^{+}}^{p^{+}} - c_{\epsilon} \, || u ||_{r( m )}^{r^{-}} \\& \geq c^{*} \, || u ||_{\mathcal{X}}^{q^{+}} - \epsilon \, c_{1}^{p^{+}} \, || u ||_{\mathcal{X}}^{p^{+}} - c_{\epsilon} \, c_{2}^{r^{-}} \, || u ||_{\mathcal{X}}^{r^{-}}.
 \end{align*}
Since, $ q^{+} < p^{+} < r^{-},$ there are two positive real numbers $\mu$ and $\nu$ such as $$ \Psi ( u ) \geq \mu > 0 \,\, \mbox{for all} \,\, u \in \mathcal{X} \,\, \mbox{with} \,\, || u ||_{\mathcal{X}} = \nu.$$
On the other hand, we affirm that there exists $u_{0} \in \mathcal{X} \backslash \overline{\mathcal{B}_{0} ( \nu )}$ such as
\begin{equation}
\Psi ( u ) < 0.
\end{equation}
Let $ \phi_{0} \in \mathcal{X} \backslash \{ 0 \},$ by $( h_{4} )$ we can choose a constant
$$ \delta > \frac{  \displaystyle \frac{1}{p^{-}} \,\int_{\mathrm{M}} | \nabla \phi_{0} |^{p( m )} \,\, dv_{g} ( m ) - \frac{1}{q^{-}} \int_{\mathrm{M}} | \nabla \phi_{0} |^{q( m )} \,\, dv_{g} ( m )}{ \displaystyle\int_{\mathrm{M}} | \phi_{0} |^{p^{+}} \,\, dv_{g} ( m )},$$
and a constant $c_{\delta} > 0$ depending on $ \delta$ such as
$$ H( m, t ) \geq \delta \, | t |^{p^{+}} \,\, \mbox{for all} \,\, | t | > c_{\delta} \,\, \mbox{and uniformly in} \,\, \mathrm{M}.$$
Let $ k > 1$ be large enough, we have
\begin{align*}
\Psi ( k \, \phi_{0} ) &= \int_{\mathrm{M}} \frac{1}{p( m )} | \nabla ( k \, \phi_{0} ) |^{p( m )} \,\, dv_{g} ( m ) + \int_{\mathrm{M}} \frac{1}{q( m )} | \nabla ( k \, \phi_{0} ) |^{q( m )} \,\, dv_{g} ( m ) \\& \hspace*{0.3cm}- \int_{\mathrm{M}} H( m, k\, \phi_{0} ) \,\, dv_{g} ( m ) \\& \geq \frac{k^{p^{+}}}{p^{-}} \int_{\mathrm{M}} | \nabla \phi_{0} |^{p( m )} \,\, dv_{g} ( m ) + \frac{k^{q^{+}}}{q^{-}} \int_{\mathrm{M}} | \nabla \phi_{0} |^{q( m )} \,\, dv_{g} ( m )\\& \hspace*{0.3cm} - \int_{\{\, | k \phi_{0} | > c_{\delta} \,\}} H( m, k \phi_{0} ) \,\, dv_{g} ( m ) - \int_{\{\, | k \phi_{0} | < c_{\delta} \,\}} H( m, k \phi_{0} ) \,\, dv_{g} ( m ) \\& \geq \frac{k^{p^{+}}}{p^{-}} \int_{\mathrm{M}} | \nabla \phi_{0} |^{p( m )} \,\, dv_{g} ( m ) + \frac{k^{q^{+}}}{q^{-}} \int_{\mathrm{M}} | \nabla \phi_{0} |^{q( m )} \,\, dv_{g} ( m )\\& \hspace*{0.3cm} - \int_{\{\, | k \phi_{0} | \leq c_{\delta} \,\}} H( m, k \phi_{0} ) \,\, dv_{g} ( m ) - \delta \, k^{p^{+}} \int_{\mathrm{M}} | \phi_{0} |^{p^{+}} \,\, dv_{g} ( m ) \\& \hspace*{0.3cm} + \delta \, \int_{\{ | k \phi_{0} | \leq c_{\delta}\}} | k \phi_{0} |^{p^{+}} \,\, dv_{g} ( m ) \\& \geq \frac{k^{p^{+}}}{p^{-}} \int_{\mathrm{M}} | \nabla \phi_{0} |^{p( m )} \,\, dv_{g} ( m ) + \frac{k^{q^{+}}}{q^{-}} \int_{\mathrm{M}} | \nabla \phi_{0} |^{q( m )} \,\, dv_{g} ( m )\\& \hspace*{0.3cm} - \delta \, k^{p^{+}} \int_{\mathrm{M}} | \phi_{0} |^{p^{+}} \,\, dv_{g} ( m ) + c_{5},
\end{align*}
which implies that $$ \Psi ( k \phi_{0} ) \longrightarrow - \infty \,\, \mbox{as} \,\, k \rightarrow + \infty.$$
Then, there exists $k_{0} > 0$ and $u_{0} = k_{0} \, \phi_{0} \in \mathcal{X}_{0} \backslash \overline{\mathcal{B}_{\nu} ( 0 )}$ such as \eqref{20} hold.\\
Thereby, proposition \ref{prop1} shows that \eqref{problem} has at least a non-trivial weak solution. This completes the proof.
\end{proof}
\begin{theorem}\label{theo6}
Assume that $( h_{1} ), \, ( h_{3} ), \, ( h_{4} ), \,( g )$ hold, we suppose that the smooth complete compact Riemannian N-manifold $( \mathrm{M}, \, g )$ has property $B_{vol} ( \lambda, \, v ).$ If $q^{-} > p^{+},$ then the problem \eqref{problem} has a sequence of weak solutions with unbounded energy.
\end{theorem}
\begin{proof}
we will divide the proof of this theorem into two steps. In the first step, we will demonstrate that the problem \eqref{problem} acquires a sequence of weak solutions $\{ \pm u_{j} \}_{j = 1}^{\infty}$ such as $$ \Psi ( \pm u_{j} ) \longrightarrow + \infty \,\, \mbox{when} \,\, j \rightarrow + \infty.$$ In the second step, we will prove that if $k$ is large enough, then there exist $ D_{k} > d_{k} > 0$ such as the assertions $( A_{1} )$ and $( A_{2} )$ are satisfied.\\
\textbf{Step 1:} The proof is based on the Fountain Theorem (given by Theorem \ref{prop2}). Indeed, from $( h_{3} ),\, \Psi $ is an even functional. And from Lemma \ref{lemma2}, $\Psi $ meets the condition $( Cer ).$\\
For that, we will use the mean value theorem in the following form: For every $\beta \in C_{+} ( \overline{\mathrm{M}} ) = \{ \beta \in C( \overline{\mathrm{M}} ), \, \beta ( m )  > 1 \,\, \forall m \in \mathrm{M} \}$ and $ u \in L^{\beta ( m )} ( \mathrm{M} ),$ there exist $\zeta \in \mathrm{M}$ such that
\begin{equation}\label{22}
\int_{\mathrm{M}} | u( m ) |^{\beta ( m )} \,\, dv_{g} ( m ) = | u |^{\beta ( \zeta )}_{\beta ( m )}.
\end{equation}
Indeed, it is easy to see that $$ \rho_{\beta ( m )} \bigg( \frac{u}{|| u ||_{\beta ( m )}} \bigg) = \int_{\mathrm{M}} \bigg( \frac{| u |}{|| u ||_{\beta ( m )}} \bigg)^{\beta ( m )} \,\, dv_{g} ( m ) = 1.$$
And, according to the mean value theorem, a positive constant $\bar{\beta} \in [ \beta^{-}, \, \beta^{+} ]$ depends on $\beta$ exists, such as
$$ \int_{\mathrm{M}} \bigg( \frac{| u |}{|| u ||_{\beta ( m )}} \bigg)^{\beta ( m )} \,\, dv_{g} ( m ) = \bigg( \frac{1}{|| u ||_{\beta ( m )}} \bigg)^{\bar{\beta}} \, \int_{\mathrm{M}} | u |^{\beta ( m )} \,\, dv_{g} ( m ).$$
Moreover, the continuity of $\beta$ ensures that there exists $ \zeta \in \mathrm{M} $ such as $ \beta ( \zeta ) = \bar{\beta}.$ Combining this fact with the above inequalities, we get \eqref{22}.\\
\textbf{Step 2:} $( A_{1} ):$ For all $ u \in Z_{k} $ such as $|| u ||_{\mathcal{X}} = d_{k}$ ( $d_{k}$ will be specified below ), by $( h_{1} ),$\eqref{22} and Proposition \ref{prop5} we obtain
\begin{align*}
\Psi ( u ) &= \int_{\mathrm{M}} \frac{1}{p( m )} \, | \nabla u( m ) |^{p( m )} \,\, dv_{g} ( m ) + \int_{\mathrm{M}} \frac{1}{q( m )} | \nabla u( m ) |^{q( m )} \,\, dv_{g} ( m )  \\& \hspace*{0.3cm}- \int_{\mathrm{M}} H( m, u( m ) ) \,\, dv_{g} ( m ) \\& \geq \bigg( \frac{1}{p^{+}} + \frac{1}{A p^{-} q^{+}} \bigg) \, || u ||_{\mathcal{X}}^{p^{-}} - c_{5}\, || u ||_{r( m )}^{r( \zeta )} - c_{6} \, || u ||_{\mathcal{X}} \,\,\, \mbox{where} \,\,
\, \zeta \in \mathrm{M}\\& \geq \begin{cases} \big( \frac{1}{p^{+}} + \frac{1}{A p^{-} q^{+}} \big) \,|| u ||_{\mathcal{X}}^{p^{-}} - c_{5} - c_{6} \, || u ||_{\mathcal{X}} & \text{if \, $|| u ||_{r( m )} \leq 1$} \\[0.3cm]
\big( \frac{1}{p^{+}} + \frac{1}{A p^{-} q^{+}} \big) \,|| u ||_{\mathcal{X}}^{p^{-}} - c_{5} ( \eta_{k} \, || u ||_{\mathcal{X}} )^{r^{+}} - c_{6} \, || u ||_{\mathcal{X}} & \text{if \, $|| u ||_{r( m )} > 1$} \end{cases} \\& \geq \bigg( \frac{1}{p^{+}} + \frac{1}{A p^{-} q^{+}} \bigg) \, || u ||_{\mathcal{X}}^{p^{-}} - c_{5} ( \eta_{k} \, || u ||_{\mathcal{X}} )^{r^{+}} - c_{6} \, || u ||_{\mathcal{X}} - c_{5} \\& \geq d_{k}^{p^{-}} \, \bigg( \frac{1}{p^{+}} + \frac{1}{A p^{-} q^{+}} - c_{5} \eta_{k}^{r^{+}} d_{k}^{r^{+} - p^{-}} \, \bigg) - c_{6} \, d_{k} - c_{5}.
\end{align*}
We fix $d_{k}$ as follows $$ d_{k} = ( r^{+} \, c_{5} \, \eta_{k}^{r^{+}} )^{\frac{1}{p^{-} - r^{+}}}.$$
Then, $$ \Psi ( u ) \geq d_{k}^{p^{-}} \, \bigg( \frac{1}{p^{+}} + \frac{1}{A p^{-} q^{+}} - \frac{1}{r^{+}} \,  \bigg) - c_{6} \, d_{k} - c_{5}.$$
According to Lemma 3.4 in \cite{zhang2015existence}. We know that $ \displaystyle \lim_{k \rightarrow + \infty} \theta ( m ) = 0.$ Then, since $ 1 < q^{+} < p^{-} \leq p^{+} < r^{+},$ we conclude that $$ d_{k} \longrightarrow + \infty \,\, \, \mbox{as} \,\,\, k \rightarrow + \infty.$$
Thus, $$ \Psi ( u ) \longrightarrow + \infty \,\,\, \mbox{as} \,\,\, || u || \rightarrow + \infty \,\,\, \mbox{with} \,\,\, u \in Z_{k}.$$
Which means that the assertion $( A_{1} )$ is verified.\\
$( A_{2} ):$ According to Remark \ref{remark1}, since $\mbox{dim} Y_{k}$ is finite, there exists a constant $ \zeta_{k} > 0,$ for all $ u \in Y_{k}$ with $ || u ||_{\mathcal{X}} $ is big enough, we obtain
\begin{align*}
\mathcal{J} ( u ) &\leq \frac{1}{p^{-}} \int_{\mathrm{M}} | \nabla u( m ) |^{p( m )} \,\, dv_{g} ( m ) + \frac{1}{q^{-}} \int_{\mathrm{M}} | \nabla u( m ) |^{q( m )} \,\, dv_{g} ( m ) \\& \leq \frac{1}{p^{-}} || u ||_{p( m )}^{p^{+}} + \frac{1}{q^{-}} || u ||_{q( m )}^{q^{+}} \\& \leq \frac{c_{7}}{p^{-}} || u ||_{\mathcal{X}}^{p^{+}} + \frac{c_{8}}{q^{-}} || u ||_{\mathcal{X}}^{q^{+}},
\end{align*}
since $p^{+} > q^{+},$ we get
\begin{equation}\label{23}
\mathcal{J} ( u ) \leq \bigg( \frac{c_{7}}{p^{-}} + \frac{c_{8}}{q^{-}} \bigg) || u ||_{\mathcal{X}}^{p^{+}} = \zeta_{k} \, || u ||_{p^{+}}^{p^{+}}.
\end{equation}
Next, according to $( h_{2} )$, there exists  $B_{k} > 0$ such as for any $| t | \geq B_{k},$ we have $$ H( m, t ) \geq 2 \, \zeta_{k} \, | t |^{p^{+}} \,\,\, \mbox{for all} \,\,\, m \in \mathrm{M}.$$ Moreover, from $( h_{1} ),$ there exists a positive $L_{k}$ such as $$ H( m, t ) \leq L_{k} \,\,\, \mbox{for all} \,\,\, ( m, t ) \in M \times [ - B_{k}, \, B_{k} ].$$
Then, for every $( m, t ) \in M \times \mathbb{R}$ we conclude that
\begin{equation}\label{24}
H( m, t ) \geq 2 \, \zeta_{k} \, | t |^{p^{+}} - L_{k}.
\end{equation}
Combining \eqref{23} and \eqref{24}, for all $ u \in Y_{k} $ such as $ || u ||_{\mathcal{X}} = D_{k} > d_{k}$ we infer that
\begin{align*}
\Psi ( u ) &= \mathcal{J} ( u ) - I( u ) \\& \leq \zeta_{k} \, || u ||_{p^{+}}^{p^{+}} - 2 \, \zeta_{k} \, || u ||_{p^{+}}^{p^{+}} + L_{k} \\& \leq - \zeta_{k} \, c_{9} \, || u ||_{\mathcal{X}}^{p^{+}} + L_{k},
\end{align*}
Hence, for $D_{k}$ large enough $ ( D_{k} > d_{k} ),$ we obtain that $$ \max_{\{ u \in Y_{k}: \, || u ||_{\mathcal{X}} = D_{k} \}} \Psi ( u ) = 0.$$
which implies that the assertion $( A_{2} )$ holds. Then by Theorem \ref{prop2}, we get the result.
\end{proof}
We end this work by giving an example of application,
\begin{example}
For $ h( m, t ) = \displaystyle \frac{t\,| t |^{\alpha( m ) - 2}}{\mbox{Log} \,( 1 + | t | )},$ if $ t \neq 0$ and $ h( m, 0 ) = 0$ with $ 2 < p^{+} \leq \alpha( m ) < r( m )$ for all $ m \in \overline{\mathrm{M}}.$  So the function $h$ doesn't satisfy the Ambrosetti-Rabinowitz condition, but it satisfies our assumptions $( h_{1} ) - ( h_{5} )$. Then, our problem \eqref{problem} becomes
$$ \begin{cases}
- \, \Delta_{p( m )} u( m ) - \Delta_{q( m )} u( m ) = \, \frac{t\,| t |^{\alpha( m ) - 2}}{\mbox{Log} \,( 1 + | t | )}  & \text{in \,\,$\mathrm{M}$ }, \\[0.3cm]
\, u \,  = \, 0  & \text{on \,$\partial \mathrm{M}$ },
\end{cases} $$
Consequently, the results corresponding to Theorems \ref{theo3} and \ref{theo6} can be achieved and still be true for the problem \eqref{problem}.
\end{example}


\section*{Funding}
This paper has been supported by the RUDN University Strategic Academic Leadership Program and P.R.I.N. 2019.

\section*{Authors’ contributions}
The authors declare that their contributions are equal.

\section*{Acknowledgments}
Firstly, the authors would like to thank Professor  G{\'o}rka Przemys{\l}aw for his support and encouragement. Secondly, to be grateful to the anonymous referees for the valuable suggestions and comments that improved the presentation's quality.


\end{document}